\documentclass{amsart}
\usepackage[utf8]{inputenc} 
\usepackage[T1]{fontenc} 

\usepackage{amsmath, amsthm, amssymb, mathrsfs}
\usepackage{enumerate}
\usepackage{tikz-cd}
\usetikzlibrary{arrows}

\usepackage{hyperref}

\usetikzlibrary{matrix,arrows,decorations.pathmorphing}

 \let\temp\phi
\let\phi\varphi
\let\varphi\temp

\newcommand{\cl}[1]{\overline{#1}}

\newcommand{\C}{\mathbb{C}}

\newcommand{\N}{\mathbb{N}}
\newcommand{\R}{\mathbb{R}}

\DeclareMathOperator{\Span}{span}

\newcommand{\calC}{\mathcal{C}}

\newcommand{\calF}{\mathcal{F}}

\newcommand{\calM}{\mathcal{M}}

\DeclareMathOperator{\conv}{conv}

\newcommand{\Angle}[1]{\left\langle #1 \right\rangle}

\DeclareMathOperator{\id}{id}

\theoremstyle{plain}
\newtheorem{lemma}{Lemma}

\newtheorem{proposition}[lemma]{Proposition}
\newtheorem{corollary}[lemma]{Corollary}

\theoremstyle{definition}

\newtheorem{example}[lemma]{Example}
\newtheorem{remark}[lemma]{Remark}
\newtheorem{definition}[lemma]{Definition}

\usepackage[doi=false, url =false , isbn = false,style=alphabetic,sorting=nyt, backend = biber, maxcitenames=50, maxnames=50]{biblatex}

\newtheorem{lettertheorem}{Theorem}
\newtheorem{lettercorollary}[lettertheorem]{Corollary}

\title[On the CPAP and the Boundary Condition for the Zero Map]{On the Completely Positive Approximation Property for Non-Unital Operator Systems and the Boundary Condition for the Zero Map}
\author{Se-Jin Kim}
\thanks{S.~Kim is supported by the Research Foundation Flanders (FWO) research project G085020N and the internal KU Leuven funds project number C14/19/088.}
\address{KU Leuven, Department of Mathematics, Leuven, Belgium}
\email{sam.kim@kuleuven.be}

\begin{document}
\maketitle 
 
\begin{abstract}
	The purpose of this paper is two-fold: firstly, we give a characterization on the level of non-unital operator systems for when the zero map is a boundary representation. As a consequence, we show that a non-unital operator system arising from the direct limit of C*-algebras under positive maps is a C*-algebra if and only if its unitization is a C*-algebra. Secondly, we show that the completely positive approximation property and the completely contractive approximation property of a non-unital operator system is equivalent to its bidual being an injective von Neumann algebra. This implies in particular that all non-unital operator systems with the completely contractive approximation property must necessarily admit an abundance of positive elements.
\end{abstract}

\section{Introduction}
	Recently the present author, in joint work with Matthew Kennedy and Nicholas Manor, initiated the study of non-unital operator systems, henceforth referred to as just \emph{operator systems}, by way of the nc convex duality of Davidson--Kennedy \cite{DavidsonKennedy2019, KKM2023}. This provides a new framework for the norm closed self-adjoint subspaces of bounded linear operators on a Hilbert space as abstractly characterized by Werner \cite{Werner} and further explored in \cite{connes2021spectral,connes2022tolerance, HumeniukKennedyManor, Ng2022}. In the nc convex algebraic geometry of \cite{KKM2023}, associated to an operator system $E$ is a geometric object $(QS(E),\delta_0)$, where $QS(E)$ is the nc convex space in the notation of \cite{HumeniukKennedyManor} corresponding to the class of completely positive and completely contractive representations of $E$ and where $\delta_0$ is a distinguished point in $QS(E)$ corresponding to the zero map $\delta_0: E \to \{0\}$. In the case of C*-algebras, extremal points correspond to irreducible representations. In particular, in \cite[Lemma 10.8]{KKM2023} it is demonstrated that a key property that the nc convex sets associated to C*-algebras satisfy that general operator systems need not necessarily satisfy is that the distingushed point $\delta_0$ is an extremal point. 
    
    On the dual side, one distinguishing feature of C*-algebras is that they always admit an abundance of positive elements, as explored in \cite[Section 8]{HumeniukKennedyManor}. In contrast, it is well known that there are operator systems where there are \emph{no} positive elements such as the set of trace zero matrices in $M_n$. The first main result shows that these two distinctions are equivalent. That is, 
\begin{lettertheorem}\label{Theorem: main}
	Let $E$ be an operator system. The following are equivalent:
	\begin{enumerate}[(1)]
		\item $E$ is approximately generated by positives, that is, $\cl{\Span}(E_+) = E$.
		\item The zero map $\delta_0:E \to \{0\}$ is an extremal point.  
	\end{enumerate}
\end{lettertheorem}
Condition $(1)$ is satisfied by many examples such as all unital and approximately unital operator systems in the sense of \cite{Ng1969} and \cite{Huang2011}, all C*-algebras, and, by the Hahn--Jordan decomposition Theorem, any operator system with a predual \cite{Ng2022}. In contrast, by \cite[Lemma 8.5]{HumeniukKennedyManor}, condition $(1)$ for finite dimensional operator systems is only satisfied if the operator system is already unital. 

Recently, \cite{courtney2023completely, courtney2023nuclearity, GoldbringSinclair} explore how one can identify a C*-algebra using the language of operator systems. The papers \cite{courtney2023completely, courtney2023nuclearity} are in particular interested in the case when an operator system is the direct limit of a sequence of completely positive contractions on C*-algebras. As a consequence of Theorem~\ref{Theorem: main}, we establish the following:
\begin{lettercorollary}\label{Corollary: limit of C*-algebras}
	Let $E$ be an operator system. Suppose that there is a collection of C*-algebras $\calM_i$ and positive maps $\rho_i: \calM_i \to E$ such that $E = \cl{\bigcup_i \rho_i(\calM_i)}$. The following are equivalent: 
	\begin{enumerate}
		\item $E$ is isomorphic to a C*-algebra. 
		\item The canonical unitization $E^\sharp$ is isomorphic to a C*-algebra.
	\end{enumerate}
\end{lettercorollary}
This result therefore allows one to characterize nuclear C*-algebras in the category of non-unital operator systems by looking at its unitization. Note by Example~\ref{Example: no positives}, Corollary~\ref{Corollary: limit of C*-algebras} is not true if we drop the assumptions on $E$. 

The final section of this paper is on a characterization of the completely positive approximation property (henceforth denoted CPAP) and the completely contractive approximation property (CCAP) for operator systems. These comprise of those operator systems for which the identity map can be approximately factorized by finite dimensional C*-algebras. See Definition~\ref{Definition: cpap} for a more precise formulation. In \cite[Theorem 3.5]{HanPaulsen}, it is shown that for unital operator systems, the CCAP and CPAP are equivalent to injectivity of the bidual. On the other hand, due to the lack of local reflexivity, \cite{kirchberg1995,EOR} demonstrate that for general operator spaces, injectivity of the bidual is not enough to get the CCAP. The second main result of this paper demonstrates that operator systems behave closer to their unital counterpart in this regard. This is therefore the appropriate generalization of nuclearity of C*-algebras and nuclearity of unital operator systems (see \cite{BrownOzawa} and the references therein for results and applications of nuclear C*-algebras), generalizing \cite[Theorem 3.5]{HanPaulsen} and refining \cite[Theorem 6.3]{HumeniukKennedyManor}. 
\begin{lettertheorem}\label{Theorem: cpap}
	Let $E$ be an operator system. The following are equivalent: 
	\begin{enumerate}
		\item $E$ has the CPAP. 
        \item $E$ has the CCAP. 
        \item $E^{**}$ is isomorphic as an operator system to an injective von Neumann algebra.
		\item $E^\sharp$ is a nuclear operator system and $E$ is approximately generated by positives.
	\end{enumerate}
\end{lettertheorem}
In particular, unlike the case of general operator spaces \cite[Theorem 4.5]{EOR}, for operator spaces that are closed under involution, local reflexivity of $E$ when its bidual is injective is automatic. Finally, we give an example of a non-unital operator system that is not a C*-algebra which exhibits the CPAP. 

\subsection*{Acknowledgements}
	The main ideas for this paper arose during the conference \emph{Noncommutativity in the North - MikaelFest}. I would like to thank the organizers Georg Huppertz, Hannes Thiel, and Eduard Vilalta for fostering an excellent environment for new ideas as well as the financial support received to attend the conference. I would also like to thank Kristin Courtney for suggesting that I write this paper. Finally, I am supported by the Research Foundation Flanders (FWO) research project G085020N and the internal KU Leuven funds project number C14/19/088.

\section{Characterization of when $\delta_0$ is an extremal point}

In this paper, an \emph{operator system} will always mean norm-closed subspace $E \subseteq B(H)$ that is closed under involution: $x^* \in E$ if $x \in E$. If $E$ additionally contains the unit $1 \in B(H)$, then $E$ is said to be \emph{unital}. This differs from the standard terminology where operator systems are always assumed to be unital. For every $d \geq 1$, the set of positive elements in $M_d(E)$ is denoted $M_d(E)_+$. As in \cite[Definition 2.11]{connes2021spectral}, the canonical unitization of an operator system is denoted $E^\sharp$. Note that when $E$ is a C*-algebra, $E^\sharp$ agrees with the usual unitization of a C*-algebra. The following terminology comes from \cite[Section 8]{HumeniukKennedyManor}.
\begin{definition}
For an operator system $E$, we say that $E$ is \emph{approximately generated by positives} if the set $E_+$ spans a dense subset of $E$.
\end{definition}

In \cite[Example 8.1]{HumeniukKennedyManor}, it is shown that there exist examples of operator systems $E$ which are approximately generated by positives, but not generated by positives. A polarization trick shows that being approximately generated by positives is equivalent to being densely spanned by the entries of the positive matrices in $M_d(E)$ for all $d$.
\begin{proposition}
	Let $E$ be an operator system. $E$ is approximately generated by positives if and only if the set 
	\begin{align*}
		E_0 := \Span\{\Angle{X \xi, \eta}: d \geq 1, \xi, \eta \in \C^d, X \in M_d(E)_+\}
	\end{align*}	
	is dense in $E$. 
\end{proposition}

\begin{proof}
Since $E_+$	is contained in $E_0$, one direction follows. Conversely, suppose that $E_0$ spans a dense subset of $E$. Let $X \in M_d(E)_+$ and fix $\xi, \eta \in \C^d$. The polarization identity gives us: 
\begin{align*}
	\Angle{X \xi,\eta} = \frac{1}{4}\sum_{k=0}^3 i^k\Angle{X (\xi + i^k \eta), (\xi + i^k \eta)} \;. 
\end{align*}
As $\Angle{X \zeta,\zeta} \in E_+$ for all $\zeta \in \C^d$, the  right hand side of the equation belongs to $\Span(E_+)$.
\end{proof}

Given an operator system $E$, the set of completely positive and completely contractive (henceforth denoted cpcc) maps $\phi: E \to B(H)$ admits an ordering called the dilation order $\preceq$. Given $\phi:E \to B(H)$ and $\psi: E \to B(K)$, we say that $\phi \preceq \psi$ if there exists an isometry $V : H \hookrightarrow K$ such that $\phi = V^* \psi V$. We say that $\phi$ is \emph{maximal in dilation order} if all dilations $\psi \succeq \phi$ are trivial: $\psi = \phi \oplus \zeta$ for some cpcc $\zeta$. See \cite{arveson2003notes} and \cite[Section 5]{DavidsonKennedy2019} for more details. 

By an extremal point of an operator system $E$, we mean an extremal point in the sense of \cite[Definition 6.1.1]{DavidsonKennedy2019} on the space $QS(E)$ described in the introduction. By \cite[Corollary 6.2.1]{DavidsonKennedy2019} and \cite[Proposition 4.4]{KKM2023}, this is equivalent to the notion of a boundary representation due to Arveson \cite{arveson2008noncommutative}. For the purposes of this paper, we appeal to \cite[Proposition 6.1.4]{DavidsonKennedy2019} to give the following alternative but equivalent definition: 
\begin{definition}
    Let $E$ be an operator system. Let $n \geq 1$ be a possibly infinite cardinal and let $\phi: E \to B(\ell^2(n))$ be a cpcc map. The map $\phi$ is said to be an \emph{extremal point} if $\phi$ is extremal in the convex set $QS_n(E)$ of cpcc maps $E \to B(\ell^2(n))$ and $\phi$ is maximal in dilation order. 
\end{definition}

The following result will be implicitly used throughout the paper. 
\begin{proposition}
	Let $E$ be an operator system. The following are equivalent: 
    \begin{enumerate}
        \item The zero map $\delta_0$ is an extremal point. 
        \item The map $\delta_0^\sharp$ is an extremal point. 
        \item $\delta_0$ is maximal in dilation order. 
        \item $\delta_0^\sharp$ is maximal in dilation order. 
        \end{enumerate}
\end{proposition}

\begin{proof}
	Since both $\delta_0$ and $\delta_0^\sharp$ are pure states $(1) \iff (3)$ and $(2) \iff (4)$ hold. That $(1)$ and $(2)$ are equivalent follows from \cite[Proposition 4.4]{KKM2023}. 
\end{proof}

\begin{proof}[Proof of Theorem \ref{Theorem: main}]
	Firstly, assume that $E$ is approximately generated by positives. Let $V: \C \hookrightarrow H$ be an isometry and let $\theta: E \to B(H)$ be a cpcc map for which $V^* \theta V = \delta_0$. Let $W: \{V(1)\}^\perp \hookrightarrow H$ be the inclusion map. Our goal is to show that $W^* \theta(x) V = 0$ and $V^* \theta(x)W = 0$ for all $x \geq 0$. Since 
	\begin{align*}
			0 \leq \theta(x) = \left[\begin{array}{cc} 0 & V^* \theta(x) W \\ W^* \theta(x) V & W^* \theta(x) W \end{array} \right]\;,
	\end{align*}
	by \cite[Lemma 3.1]{PaulsenBook}, it follows that $W^* \theta(x) V = 0$ and $V^* \theta(x) W = 0$.  
	
	By contraposition, assume that $E$ is not approximately generated by positives. Let $E_0$ be the operator subspace generated by the terms of the form $\Angle{X\xi, \eta}$ for $X \in M_d(E)_+$ and $\xi,\eta \in \C^d$. Fix a completely contractive map $f: E \to \C$ such that $f_{E_0} = 0$ but for which for some $y \in E$, $f(y) \neq 0$. Such a map exists by composing any non-trivial completely contractive map on $E/E_0$ with the quotient map $E \to E/E_0$. Define 
	\begin{align*}
		\theta: E \to M_2: x \mapsto \left[\begin{array}{cc} 0 & f(x) \\ f(x^*)^* & 0 \end{array}\right]\;.
	\end{align*}
	Our goal is to show that $\theta$ is a cpcc map. Firstly, observe that as $f$ is completely contractive, that $\theta$ is a completely contractive *-preserving map as well. For complete positivity, let $X \in M_d(E)_+$ and let $\xi, \eta \in \C^d$. A calculation shows that 
\begin{align*}
	\Angle{\theta^{(d)}(X)\xi \oplus \eta, \xi \oplus \eta} = f(\Angle{X\eta,\xi}) + f(\Angle{X\xi, \eta}) = 0 \;. 
\end{align*}
Thus, $\theta$ is a non-trivial cpcc dilation of $\delta_0$, showing $\delta_0$ is not maximal. 
\end{proof}

As a corollary, we have a refinement of \cite[Lemma 10.8]{KKM2023}.
\begin{corollary}
	Let $E$ be an operator system. The following are equivalent:
	\begin{enumerate}[(1)]
		\item $E$ is isomorphic to a C*-algebra. 
		\item $E^\sharp$ is isomorphic to a C*-algebra and $E$ is generated by positives, that is, $\Span(E_+) = E$. 
		\item $E^\sharp$ is isomorphic to a C*-algebra and $E$ is approximately generated by positives. 
		\item $E^\sharp$ is isomorphic to a C*-algebra and the zero map is a *-homomorphism. 
	\end{enumerate}
\end{corollary}

\begin{proof}
Note that $(1) \implies (2)$ by the functional calculus. That $(2) \implies (3)$ is immediate. The implication $(3) \implies (4)$ follows from Theorem~\ref{Theorem: main}, using the fact that maximal dilations correspond to *-homomorphisms on a C*-algebra. Finally, by \cite[Lemma 10.8]{KKM2023}, $(1)$ and $(4)$ are equivalent. 
\end{proof}

\begin{example}\label{Example: no positives}
For an example where $E^\sharp$ is a C*-algebra but $E$ is not, consider the operator system
\begin{align*}
	A([-1,1],0) = \left\{[-1,1] \xrightarrow{f} \C: f\text{ is affine and }f(0) = 0\right\} \subseteq C([-1,1])\;.
\end{align*}
By \cite[Corollary 4.7]{KKM2023}, the unitization admits the isomorphism $A([-1,1],0)^\sharp = A([-1,1]) \cong \C^2$, which is a C*-algebra. However, $A([-1,1],0)$ admits no positive elements. Therefore, $A([-1,1],0)$ cannot be a C*-algebra. 
\end{example}

\begin{proposition}\label{Proposition: cpc limits}
	Let $E$ be an operator system. Suppose that we have a collection of C*-algebras $\calM_i$ and positive maps $\rho_i: \calM_i \to E$ such that $E = \cl{\bigcup_i \rho_i(\calM_i)}$. The operator system $E$ is approximately generated by positives. 
\end{proposition}

\begin{proof}
	Since for any $x \in \calM_i$, $x = \sum_{k=0}^3 i^k x_k$ for $x_k \in (\calM_i)_+$, it follows that $\rho_i(x)$ is a linear combination of positive elements. Since $\bigcup_i \rho_i(\calM_i)$ is dense in $E$, the result follows. 
\end{proof}

\begin{proof}[Proof of Corollary \ref{Corollary: limit of C*-algebras}]
By Theorem~\ref{Theorem: main}, it suffices to show that $\Span(E_+)$ is dense in $E$. This is the content of Proposition~\ref{Proposition: cpc limits}.  
\end{proof}

Theorem~\ref{Theorem: cpap} tells us that approximate positive generation of the operator system is a necessary criterion for the CPAP. The next proposition demonstrates that this phenomenon is not typical. 
\begin{proposition}
    Suppose that $E$ is a unital operator system generating a C*-algebra $A$. Let $\phi$ be a state on $E$ with kernel $F$. If the non-unital operator system $F$ is approximately generated by positives then $\phi$ extends to a *-homomorphism on $A$. In particular, if $A$ admits no one-dimensional representations, then $F$ is not approximately generated by positives. 
\end{proposition}

\begin{proof} 
    Suppose that $F$ is approximately generated by positives. By Theorem~\ref{Theorem: main}, this means that $\delta_0^\sharp: F^\sharp \to \C$ is maximal in dilation order. Let $q: F^\sharp \to E$ be the ucp bijection induced by the inclusion $F \subseteq E$ and note that $\delta_0^\sharp = \phi \circ q$. We claim that $\phi$ is maximal in dilation order. To this end, fix a ucp map $\psi: E \to B(H)$ as well as an isometry $V: \C \hookrightarrow H$ such that $\phi(x) = V^* \psi(x) V$ for all $x \in E$. Let $W: \{V(1)\}^\perp \hookrightarrow H$ be the inclusion map so that we have the block $2 \times 2$-matrix decomposition 
    \begin{align*} 
        \psi(x) = \left[\begin{array}{cc} \phi(x) & V^*\psi(x)W \\ W^*\psi(x)V & W^*\psi(x)W \end{array}\right]
    \end{align*}
    for all $x \in E$. By precomposing with $q$, we get the identity 
    \begin{align*} 
        \psi\circ q(x) = \left[\begin{array}{cc} \delta_0^\sharp(x) & V^*\psi\circ q(x)W \\ W^*\psi\circ q(x)V & W^*\psi\circ q(x)W \end{array}\right]\;.
    \end{align*} 
    That is, $\psi \circ q$ is a dilation of $\delta_0^\sharp$. Since $\delta_0^\sharp$ is maximal in dilation order, this means that $V^*\psi \circ q(x)W = 0$ and $W^* \psi \circ q(x)V = 0$. Since $q$ is a bijection, this implies that $\psi$ is a trivial dilation of $\phi$. Thus, $\phi$ is maximal in dilation order. By \cite[Proposition 2.2]{arveson2003notes}, $\phi$ extends to a *-homomorphism on $A$. The \emph{In particular} follows. 
\end{proof}

\section{The Completely Positive Approximation Property}

The goal of this section is to prove Theorem~\ref{Theorem: cpap}. 
\begin{definition}\label{Definition: cpap}
	Let $E$ be an operator system.
    \begin{enumerate}
        \item The operator system $E$ has the \emph{completely positive approximation property} (CPAP) if there exists a net of cpcc maps $\psi_i : E \to M_{n_i}$ and $\phi_i: M_{n_i} \to E$ for $n_i \geq 1$ such that $\lim_{i \to \infty} \phi_i \circ \psi_i = \id_E$, where the limit is taken in the point-norm topology. 

        \item The operator system $E$ has the \emph{completely contractive approximation property} (CCAP) if there exists a net of completely contractive maps $\psi_i : E \to M_{n_i}$ and $\phi_i: M_{n_i} \to E$ for $n_i \geq 1$ such that $\lim_{i \to \infty} \phi_i \circ \psi_i = \id_E$, where the limit is taken in the point-norm topology. 
        
        \item The operator system $E$ is said to be \emph{locally reflexive} if whenever $F$ is a finite dimensional operator space and $\psi:F \to E^{**}$ is a completely contractive map, for all $\epsilon > 0$ and $\chi \subseteq E^*$ finite, there is a completely bounded map $\psi_{\epsilon,\chi}: F \to E$ with $\|\psi_{\epsilon,\chi}\|_{cb} \leq 1 + \epsilon$ such that 
        \begin{align*} 
            \Angle{f, \psi_{\epsilon,\chi}(x)} = \Angle{f, \psi(x)}
        \end{align*} 
        for all $x \in F$ and $f \in \chi$. 
    \end{enumerate}
\end{definition}

In \cite[Theorem 4.5]{EOR}, it is shown that $E$ has the CCAP if and only if $E^{**}$ is an injective von Neumann algebra and $E$ is locally reflexive. We first show that local reflexivity when $E^{**}$ is injective is automatic for operator systems. For what follows, we adopt the terminology of \cite{KirchbergWassermann} and say that a \emph{C*-system} is an operator system $E$ for which $E^{**}$ is a von Neumann algebra. 
\begin{lemma}\label{Lemma: Local Reflexivity}
	Let $E$ be a C*-system. If the unitization $E^\sharp$ is locally reflexive then $E$ is locally reflexive. In particular, if $E^{**}$ is an injective unital operator system then $E$ is locally reflexive. 
\end{lemma}

\begin{proof}
	Let $F$ be a finite dimensional operator space and let $\psi: F \to E^{**}$ be a completely contractive map. Let $\epsilon > 0$ and fix $\chi \subseteq E^*$ finite. Assume by extending $\chi$ if necessary that $\delta_0 \in \chi$. By the proof of \cite[Theorem 6.3]{HumeniukKennedyManor}, we know that $(E^\sharp)^{**} = (E^{**})^\sharp = E^{**} \oplus \C$. Let $\iota: E^{**} \hookrightarrow (E^\sharp)^{**}: x \mapsto x \oplus 0$ be the canonical embedding. Since $\iota \circ \psi: F \to (E^\sharp)^{**}$ is a completely contractive map, there is a map $\phi: F \to E^\sharp$ with $\|\phi\|_{cb} \leq 1 + \epsilon$ such that 
	\begin{align*}
		\Angle{f^\sharp, \phi(x)} = \Angle{f^\sharp, \iota \circ \psi(x)} = \Angle{f, \psi(x)}
	\end{align*} 
	for all $f \in \chi$ and $x \in F$. In particular, $\Angle{\delta_0^\sharp, \phi(x)} = 0$ for all $x \in F$. That is, $\phi(M_n) \subseteq E$. 

    If $E^{**}$ is unital and injective, then by \cite[Theorem 3.1]{ChoiEffros}, $E^{**}$ is a C*-algebra. By \cite[Corollary 3.6]{HanPaulsen}, $E^\sharp$ is locally reflexive, from which local reflexivity of $E$ follows. 
\end{proof}

For the next result, we set the following notation: if $E$ and $F$ are operator systems, then $CP(E,F)$ denotes the set of completely positive maps with domain $E$ and codomain $F$. The following is essentially \cite[Lemma 2.1 (iii)]{BlecherMagajna}.
\begin{lemma}\label{Lemma: weak* density of positives}
    Let $E$ be an operator system and let $n \geq 1$. The point-weak*-closure of $CP(M_n,E)$ is $CP(M_n,E^{**})$. 
\end{lemma}

\begin{proof} 
    By Choi's Theorem \cite[Theorem 3.14]{PaulsenBook}, for any operator system $X$, there is a correspondence between $CP(M_n,X)$ and $M_n(X)_+$ such that the point-weak-topology corresponds to the weak-topology on $M_n(X)$. Thus, it suffices to show that the weak*-closure of $M_n(E)_+$ is $M_n(E)^{**}_+$. For notational convenience, we set $n = 1$. Suppose by way of contradiction that there is an $x \in E^{**}_+$ that does not belong to the weak*-closure of $E_+$. By the Hahn-Banach separation theorem, there is linear map $\phi: E \to \C$ and $\alpha > 0$ such that 
    \begin{align*} 
        Re(\phi)(a) < \alpha < Re(\phi)(x)\;
    \end{align*} 
    for all $a \in E_+$. Since $\lambda E_+ \subseteq E_+$ for all $\lambda > 0$, we get that $Re(\phi)(a) \leq 0$. In particular, setting 
    \begin{align*} 
        f: E \to \C: x \mapsto -\phi(x) - \cl{\phi(x^*)}\;, 
    \end{align*} 
    we get a positive functional on $E$ with $\Angle{f, x} < -\alpha < 0$. This contradicts the definition of $E^{**}_+$. 
\end{proof} 

The following lemma is the main difficulty in the proof of Theorem \ref{Theorem: cpap}. 
\begin{lemma}\label{Lemma: density argument}
    Let $E$ be a C*-system. If $\phi: M_n \to E^{**}$ is a ucp map, then there is a net $\phi_i: M_n \to E$ of cpcc maps such that $\phi_i$ converges to $\phi$ in the point-weak* topology. 
\end{lemma}

\begin{proof}
    By \cite[Proposition 2.8]{Huang2011} along with the proof of \cite[Proposition 1]{Ng1969}, the set 
    \begin{align*} 
        \Lambda = \{x \in E_+ : \|x\| < 1 \} 
    \end{align*}
    is convex, directed under the ordering on $E_+$, and, if we define the net $e: \Lambda \to E$ by the inclusion map, then we have 
    \begin{align*} 
        \lim_\lambda e_\lambda = 1
    \end{align*} 
    where the limit is taken in the weak*-topology. 

    By \cite[Lemma 6.2 and Theorem 6.3]{HumeniukKennedyManor}, we know that $C^*_{env}(E^\sharp)^{**} = (E^\sharp)^{**} = E^{**} \oplus \C$. Furthermore, \cite[Theorem 6.3]{HumeniukKennedyManor} tells us that $\delta_0$ is an extremal point. Thus, $\delta_0^\sharp$ is an irreducible representation on $C^*_{env}(E^\sharp)$. Let $A = \ker \delta_0^\sharp$ and observe that $A^{**} = E^{**}$. 

    We represent $A^{**}$ by its standard representation in $B(H)$. In this way, we may treat $E^{**}$ with the weak*-topology as the $\sigma$-WOT on $B(H)$. In particular, due to the fact that for all $\lambda \in \Lambda$, $\sqrt{e_\lambda} \geq e_\lambda$, we get $\lim_\lambda \sqrt{e_\lambda} = 1$ in the SOT.

    Consider for all $\lambda \in \Lambda$ the completely positive map 
    \begin{align*}
        \psi_\lambda : M_n \to E^{**}: \rho \mapsto \sqrt{e_\lambda}\psi(\rho) \sqrt{e_\lambda}\;.
    \end{align*} 
    Note that $\|\psi_\lambda\|_{cb} = \|\psi_\lambda(1)\| = \|e_\lambda\| < 1$, and thus $\psi_\lambda$ is also completely contractive. 
    
    Set $\Sigma$ to be the directed set consisting of pairs $(\epsilon, \calF)$, where $\epsilon > 0$ and $\calF \subseteq M_n \times E^*$ finite with the ordering $(\epsilon_1, \calF_1) \leq (\epsilon_2, \calF_2)$ if $\epsilon_1 \geq \epsilon_2$ and $\calF_1 \subseteq \calF_2$. By Lemma~\ref{Lemma: weak* density of positives}, for $\lambda \in \Lambda$ fixed, there are completely positive maps 
    \begin{align*} 
        \psi_{\lambda, \epsilon, \calF}: M_n \to E
    \end{align*} 
    for each $(\epsilon, \calF) \in \Sigma$ such that for all $(\rho, f) \in \calF$, we have 
    \begin{align*} 
        |\Angle{f, \psi_{\lambda,\epsilon,\calF}(\rho) - \psi_\lambda(\rho)}| < \epsilon \;.
    \end{align*} 
    That is, the net $(\psi_{\lambda,\epsilon,\calF})_{(\epsilon,\calF) \in \Sigma}$ converges in the point-weak*-topology to $\psi_\lambda$. The maps $\psi_{\lambda, \epsilon, \calF}$ need not be completely bounded, but since for all $(\epsilon_0, \calF_0) \in \Sigma$ fixed, the weak-closure and norm closure of the set 
    \begin{align*} 
        \conv\{\psi_{\lambda, \epsilon, \calF}(1): (\epsilon, \calF) \geq (\epsilon_0, \calF_0)\} 
    \end{align*} 
    are equal, by replacing $\psi_{\lambda, \epsilon_0, \calF_0}$ by an element of  
    \begin{align*} 
        \calC_{\epsilon_0,\calF_0} := \conv\{\psi_{\lambda, \epsilon,\calF}: (\epsilon,\calF) \geq (\epsilon_0, \calF_0) \}
    \end{align*} 
    if necessary, we may assume that $\psi_{\lambda, \epsilon,\calF}(1)$ converges in norm to $e_\lambda$. Since $e_\lambda$ is a strict contraction, this means that we may assume that $\|\psi_{\lambda, \epsilon_0, \calF_0}(1)\| \leq 1$ by replacing $\psi_{\lambda,\epsilon_0,\calF_0}$ by an appropriate element of $\calC_{\epsilon_0,\calF_0}$. 

    Let $\Lambda \times \Sigma$ be a directed set with the ordering given by $(\lambda_1, \epsilon_1, \calF_1) \leq (\lambda_2, \epsilon_2, \calF_2)$ if $\lambda_1 \leq \lambda_2$ and $(\epsilon_1, \calF_1) \leq (\epsilon_2, \calF_2)$. We claim that $(\psi_{\lambda, \epsilon, \calF})_{\lambda, \epsilon, \calF}$ converges in the point-weak*-topology to $\psi$. 

    Fix $\rho \in M_n$ and $\epsilon_0 > 0$. Since the net $(\|\psi_{\lambda, \epsilon, \calF}(\rho)\|)_{\lambda, \epsilon, \calF}$ is uniformly bounded by $\|\rho\|$, the weak*-topology on $E^{**}$ agrees with the WOT on $B(H)$. An application of the triangle inequality shows that $\lim_\lambda \psi_\lambda(\rho) = \psi(\rho)$ in the SOT. Fix a unit vector $h \in H$. For all $(\lambda, \epsilon, \calF) \in \Lambda \times \Sigma$, 
    \begin{align*} 
        |\Angle{(\psi(\rho) - \psi_{\lambda, \epsilon, \calF})h,h}| &\leq |\Angle{(\psi(\rho) - \psi_\lambda(\rho))h,h}| + |\Angle{(\psi_\lambda(\rho) - \psi_{\lambda, \epsilon, \calF}(\rho))h,h}|\\
        &\leq \|(\psi(\rho)-\psi_\lambda(\rho))h\| + |\Angle{(\psi_\lambda(\rho) - \psi_{\lambda, \epsilon, \calF}(\rho))h,h}|\;.
    \end{align*} 
    Let $\lambda_0 \in \Lambda$ be such that $\|(\psi(\rho) - \psi_\lambda(\rho))h\| < \epsilon_0/2$ for all $\lambda \geq \lambda_0$. As well, let $\calF_0 = \{(\rho, \Angle{\cdot h,h})\}$. For all $(\epsilon, \calF) \geq (\epsilon_0/2,\calF_0)$, 
    \begin{align*} 
        |\Angle{(\psi_\lambda(\rho) - \psi_{\lambda, \epsilon_0, \calF}(\rho))h,h}| < \epsilon_0/2 \;. 
    \end{align*}
    Thus, for all $(\lambda, \epsilon, \calF) \geq (\lambda_0, \epsilon_0/2, \calF_0)$, 
    \begin{align*} 
        |\Angle{(\psi(\rho) - \psi_{\lambda, \epsilon, \calF}(\rho))h,h}| < \epsilon_0 \;. 
    \end{align*} 
    and this gives $\lim_{(\lambda,\epsilon,\calF)}\psi_{\lambda, \epsilon, \calF}(\rho) = \psi(\rho)$ in the WOT. 
\end{proof}

\begin{proof}[Proof of Theorem~\ref{Theorem: cpap}]
That $(1)$ implies $(2)$ is immediate. That $(2)$ and $(3)$ are equivalent is due to Lemma~\ref{Lemma: Local Reflexivity} and \cite[Theorem 4.5]{EOR}. That $(3)$ and $(4)$ are equivalent is \cite[Theorem 6.3]{HumeniukKennedyManor}, \cite[Theorem 3.1]{HanPaulsen}, and Theorem~\ref{Theorem: main}, along with the fact that $(E^\sharp)^{**} = E^{**} \oplus \C$ is an injective von Neumann algebra. It remains to show that $(4)$ implies $(1)$. Since $E^{**}$ is an injective von Neumann algebra, it is semidiscrete. In particular, for all $\calF \subseteq E \times E^*$ finite and $\epsilon > 0$, there is an $n \geq 1$ and ucp maps 
\begin{align*} 
    E^{**} \xrightarrow{\phi} M_n \xrightarrow{\psi} E^{**}
\end{align*} 
such that $|\Angle{f, \psi\circ \phi(\rho) - \rho}|<\epsilon/2$ for all $(\rho,f) \in \calF$. By Lemma~\ref{Lemma: density argument}, there is a net of cpcc maps 
\begin{align*} 
    \left(M_n \xrightarrow{\psi_i} E\right)_{i \in I} 
\end{align*} 
such that $\lim_i \psi_i = \psi$ in the point-weak*-topology. In particular, there is an $i_0 \in I$ such that for all $i \geq i_0$ we have 
\begin{align*} 
    |\Angle{f, \psi_i \circ \phi (\rho) - \psi \circ \phi(\rho)}| < \epsilon/2\;.
\end{align*} 
By the triangle inequality, we get that 
\begin{align*} 
    |\Angle{f, \psi_{i_0} \circ \phi(\rho) - \rho}| < \epsilon \;. 
\end{align*} 
Thus, the identity map on $E$ is the point-weak limit of maps which factor into a finite dimensional C*-algebra. By \cite[Lemma 2.3.4 and Lemma 2.3.6]{BrownOzawa}, this implies that $\id_E$ is the point-norm limit of maps which factor into a finite dimensional C*-algebra. 
\end{proof}

\begin{remark}
    Note that approximate generation by positives in condition (4) of Theorem~\ref{Theorem: cpap} can be replaced with positive generation. This is because, by \cite[Proposition 2.8]{Huang2011}, if $E^{**}$ is a von Neumann algebra, then for all $x \in E$ self-adjoint, there is some $e \in E_+$ such that 
    \begin{align*}
        -e \leq x \leq e\;. 
    \end{align*}
    In particular, $x = \frac{1}{2}(e + x) - \frac{1}{2}(e -x )$, and $e \pm x \geq 0$. 
\end{remark}

The next result shows that the non-unital analogue of \cite[Theorem 4.2]{HanPaulsen} does not have the CPAP. 
\begin{corollary}\label{Corollary: Han-Paulsen Example}
    Let $E = \cl{\Span}\{e_{i,j}: (i,j) \neq (1,1)\} \subseteq B(\ell^2(\N))$, where $e_{i,j}$ denotes the canonical matrix units in $B(\ell^2(\N))$. The operator system $E$ does not have the CPAP. 
\end{corollary}

\begin{proof}
    Let $n > 1$ and let $T \in E_+$. Observe that 
    \begin{align*}
        0 \leq \left[\begin{array}{cc} \Angle{T e_1, e_1} & \Angle{T e_1, e_n} \\ \Angle{T e_n, e_1} & \Angle{T e_n, e_n } \end{array} \right] = \left[\begin{array}{cc} 0 & \Angle{T e_1, e_n} \\ \Angle{T e_n, e_1} & \Angle{T e_n, e_n } \end{array} \right] \;. 
    \end{align*}
    Since the determinant of this positive matrix must be non-negative, we get that 
    \begin{align*}
        |\Angle{Te_1, e_n}|^2 = \Angle{Te_1,e_n}\Angle{T e_n, e_1} \leq 0\;.
    \end{align*}
    That is, $\Angle{Te_1,e_n} = \Angle{Te_n, e_1} = 0$ for all $n \geq 1$. It follows that $dist(e_{1,2}, \cl{\Span}(E_+)) = 1$, and thus $E$ is not approximately generated by positives. By Theorem~\ref{Theorem: cpap}, $E$ cannot have the CPAP. 
\end{proof}

Finally, we give an example of a non-unital operator system with the CPAP that is not a C*-algebra. 
\begin{proposition}\label{Proposition: Han-Paulsen}
    Let $E \subseteq B(H)$ be a unital operator system with the CPAP. The spatial tensor product 
    \begin{align*}
        E \otimes c_0(\N) := \cl{\Span}\{x \otimes f: x \in E, f \in c_0(\N)\} \subseteq B(H \otimes \ell^2(\N))
    \end{align*}
    has the CPAP. Furthermore, if $E \otimes c_0(\N)$ is a  C*-algebra, then $E$ is a C*-algebra. 
\end{proposition}

\begin{proof}
    By Theorem~\ref{Theorem: cpap}, it suffices to show that $E \otimes c_0(\N)$ has the CCAP. Let $E \xrightarrow{\phi_i} M_n \xrightarrow{\psi_i} E$ and $c_0(\N) \xrightarrow{\phi'_j} M_n \xrightarrow{\psi'_j} c_0(\N)$ be nets of completely contractive maps such that $\lim_i \phi_i \circ \phi_i = \id_E$ and $\lim_j \psi'_j \circ \phi'_j = \id_{c_0(\N)}$. By \cite[Section 2.1]{PisierBook}, we know that the usual spatial tensor product of completely contractive maps $\phi_i \otimes \phi'_j$ and $\psi_i \otimes \psi'_j$ are both completely contractive. For all $x \in E$ and $f \in c_0(\N)$, we have 
    \begin{align*} 
        \lim_{(i,j)} (\psi_i \otimes \psi'_j) \circ (\phi_i \otimes \phi'_j) (x \otimes f) = \lim_{(i,j)} (\psi_i \circ \phi_i)(x) \otimes (\psi'_j \circ \phi'_j)(f) = x \otimes f
    \end{align*} 
    from which the CCAP of $E \otimes c_0(\N)$ follows. 

    Suppose that $E \otimes c_0(\N)$ is a C*-algebra and let $A \subseteq B(H)$ denote the C*-algebra generated by $E$ in $B(H)$. Note that $E \otimes c_0(\N) \subseteq A \otimes c_0(\N)$. Our first claim is that this inclusion is an embedding in the sense that the canonical unital extension $(E \otimes c_0(\N))^\sharp \to (A \otimes c_0(\N))^\sharp$ is a complete order embedding. By definition of the ordering on the unitization, it suffices to show that for all $d \geq 1$, a quasistate on $M_d(E \otimes c_0(\N))$ extends to one on $M_d(A \otimes c_0(\N))$. Since $M_d(A \otimes c_0(N)) = M_d(A) \otimes c_0(\N)$, we may assume $d = 1$ for notational convenience. Let $\phi: E \otimes c_0(\N) \to \C$ be a quasistate. For $n \geq 1$, let $e_n \in c_0(\N)$ denote the characteristic function at $n$. That is, $e_n: m \mapsto \delta_{n,m}$. Let $\phi_n: E \to \C: x \mapsto \phi(x \otimes e_n)$. Since $\phi_n$ is a cpcc map on a unital operator system, by Arveson's extension theorem, there is an extension to a cpcc map $\psi_n: A \to \C$. Using the identification $c_0(\N,A) = A \otimes c_0(\N)$, setting $A_{sa} = \{a \in A: a^* = a\}$, let 
    \begin{align*}
        \psi_0 : c_0(\N, A_{sa}) \to \R: f \mapsto \sum_{n \geq 1} \psi_n(f(n))\;. 
    \end{align*}
    To see that the series on the right hand side converges, note that for all $n \geq 1$, $-\|f(n)\| 1 \leq f(n) \leq \|f(n)\| 1$ and that the function $F: \N \to \C: n \mapsto \|f(n)\|$ is in $c_0(\N)$. Since $\psi_n$ is positive, $-\|f(n)\|\psi_n(1) \leq \psi_n(f(n)) \leq \|f(n)\|\psi_n(1)$. In particular, $|\psi_n(f(n))| \leq \|f(n)\|\psi_n(1)$. Since $F$ belongs to $c_0(\N)$, we have 
    \begin{align*}
        \sum_{n \geq 1} |\psi_n(f(n))| \leq \sum_{n \geq 1}\|f(n)\|\psi_n(1) = \sum_{n \geq 1} \|f(n)\|\phi_n(1) = \phi(F) < \infty 
    \end{align*}
    giving us absolute convergence. Define 
    \begin{align*}
        \psi: c_0(\N,A) \to \C: f \mapsto \psi_0(Re(f)) + i \psi_0(Im(f))\;. 
    \end{align*}
    Note that since $\psi_0(f) = \phi(f)$ for all $f \in c_0(\N,E_{sa})$ that $\psi$ is an extension of $\phi$. For any $f \in c_0(\N,A)_+$, since $\psi_n(f(n)) \geq 0$ for all $n \geq 1$, $\psi(f) \geq 0$. It remains to see that $\psi$ is contractive. For all $n \geq 1$, let $\chi_n \in c_0(\N)$ denote the characteristic function on $\{1,2,\ldots, n\}$. The sequence $\chi_n \otimes 1 \in c_0(\N,A)$ forms an approximate unit for this C*-algebra. In particular, 
    \begin{align*}
        \|\psi\|_{cb} = \sup_{n \geq 1} \|\psi(\chi_n \otimes 1)\| = \sup_n \|\phi(\chi_n \otimes 1)\| \leq 1\;,  
    \end{align*}
    proving the claim. 

    If $E \otimes c_0(\N)$ is a C*-algebra then it is a minimal C$^*$-cover in the sense of \cite[Definition 6.5 (4)]{KKM2023}. In particular, since the inclusion $E \otimes c_0(\N) \hookrightarrow A \otimes c_0(\N)$ is an embedding, there is a *-homomorphism $\pi: A \otimes c_0(\N) \to E \otimes c_0(\N)$ such that the diagram 
    \begin{center} 
        \begin{tikzcd} 
            & A \otimes c_0(\N) \arrow[dr, "\pi"] & \\
           E \otimes c_0(\N) \arrow["\subseteq",ur] \arrow["\id",rr] & & E \otimes c_0(\N)
        \end{tikzcd}
    \end{center}
    commutes. For clarity, we denote by $*$ the multiplication on $E \otimes c_0(\N)$ and by $\cdot$ the multiplication on $A \otimes c_0(\N)$. Note that for all $n,m \geq 1$,
    \begin{align*} 
        (1 \otimes e_m) * (1 \otimes e_n) = \pi(1 \otimes e_m \cdot 1 \otimes e_n) = \pi(1 \otimes e_m)\delta_{n,m} = (1 \otimes e_m)\delta_{n,m}
    \end{align*} 
    That is, the $(1 \otimes e_m)$ form orthogonal projections in $E \otimes c_0(\N)$. Similarly, for all $x \in E$ and $n \geq 1$, 
    \begin{align*} 
        (x \otimes e_n) * (1 \otimes e_n) = \pi((x \otimes e_n) \cdot (1 \otimes e_n)) = \pi(x \otimes e_n) = x \otimes e_n \;. 
    \end{align*}
    For all $n \geq 1$, consider the cpcc map $\delta_n: c_0(\N,E) \to E : f \mapsto f(n)$. Since $1 \otimes e_1$ is in the multiplicative domain of $\delta_n$, 
    \begin{align*} 
        \delta_n((x \otimes e_1) * (y \otimes e_1)) &= \delta_n((x \otimes e_1) * (y \otimes e_1) * (1 \otimes e_1)) \\ 
        &= \delta_n((x \otimes e_1)* (y \otimes e_1)) \delta_{n,1} \;. 
    \end{align*}
    It follows that $(x \otimes e_1) * (y \otimes e_1) = \delta_1((x \otimes e_1)*(y \otimes e_1)) \otimes e_1$. 

    Define a multiplication $\circ$ on $E$ by $x \circ y = \delta_1((x \otimes e_1)*(y \otimes e_1))$. The above calculations tell us that for all $n \geq 1$ and $X,Y \in M_n(E)$, we have 
    \begin{align*} 
        (X \otimes e_1) * (Y \otimes e_1) = (X \circ Y) \otimes e_1 \;,
    \end{align*} 
    where $\circ$ on $M_n(E)$ is defined by the identity $(A \otimes x) \circ (B \otimes y) := (A \cdot B) \otimes x \circ y$ for all $x,y \in E$ and $A,B \in M_n$. It follows that $(E,\circ)$ is a *-algebra. Since for all $x \in E$,
    \begin{align*} 
        \|x\|^2 = \|x \otimes e_1\|^2 = \|(x^* \otimes e_1)* (x \otimes e_1)\| = \|(x^* \circ x) \otimes e_1\| = \|x^* \circ x\|\;,
    \end{align*}
    we have the C*-identity. Similarly, for all $x,y \in E$, $\|x \circ y\| \leq \|x\|\|y\|$. Finally, it remains to show that the identity map $ E \to (E,\circ)$ is a unital complete order isomorphism. For all $n \geq 1$ and $X \in M_n(E)$, we know that $X \geq 0$ if and only if $X \otimes e_1 \geq 0$. Since $E \otimes c_0(\N)$ is a C*-algebra, this is equivalent to the existence of $f \in M_n(E) \otimes c_0(\N)$ such that $f^* * f = X \otimes e_1$. Notice in particular that 
    \begin{align*} 
        (1 \otimes e_1) * f^* * f * (1 \otimes e_1) = X \otimes e_1 
    \end{align*} 
    and that $f * (1 \otimes e_1) = F \otimes e_1$ for some $F \in M_n(E)$. Our identity reduces to
    \begin{align*} 
        X \otimes e_1 = (F^* \otimes e_1) * (F \otimes e_1) = (F^* \circ F) \otimes e_1 \;. 
    \end{align*} 
    Therefore, $X \geq 0$ in $M_n(E)$ if and only if $X = F^* \circ F$ for some $F \in M_n(E)$, completing the proof. 
\end{proof}

\begin{corollary}
    Let $E$ be a unital operator system with the CPAP such that $E$ is not a C*-algebra, such as the operator system presented in \cite[Theorem 4.2]{HanPaulsen}. The operator system $E \otimes c_0(\N)$ has the CPAP but is not a C*-algebra. 
\end{corollary}

\begin{proof}
    This is an application of Proposition~\ref{Proposition: Han-Paulsen} to $E$. 
\end{proof}

\printbibliography

\end{document}